\documentclass[11pt,leqno]{amsart}
\usepackage[english]{babel}
\usepackage{amsfonts,amssymb,graphicx,amscd}
\usepackage{amsmath,amsthm}
\usepackage{hyperref}
\usepackage{mathrsfs}
\allowdisplaybreaks

\newtheorem{thm}{Theorem}[section]

\newtheorem{lem}[thm]{Lemma}
\newtheorem{cor}[thm]{Corollary}

\newtheorem{prop}[thm]{Proposition}

\theoremstyle{definition}
\newtheorem{ex}[thm]{Example} 

\newtheorem{defn}[thm]{Definition}

\newtheorem{rem}[thm]{Remark}

\numberwithin{equation}{subsection}
\numberwithin{thm}{section}

\newcommand{\Hom}{\text{\rm Hom}}

\newcommand{\End}{\operatorname{End}}

\newcommand{\ext}{\operatorname{ext}}
\newcommand{\Ext}{\operatorname{Ext}}

\newcommand{\Z}{\mathbb Z}

\newcommand{\A}{\mathcal A}

\newcommand{\cO}{\mathcal O}
\newcommand{\C}{\mathbb C}

\newcommand{\fg}{{\mathfrak{g}}}

\newcommand{\tothewall}{T^{\mu}_{\lambda}}
\newcommand{\fromthewall}{T_{\mu}^{\lambda}}

\newcommand{\hd}{\operatorname{hd}}
\newcommand{\rad}{\operatorname{rad}}
\newcommand{\soc}{\operatorname{soc}}
\newcommand{\E}{\mathcal E}
\newcommand{\stab}{\operatorname{Stab}}

\newcommand{\ch}{\operatorname{ch}}

\let\tothe\tothewall
\let\fromthe\fromthewall
\newcommand{\D}{\mathcal D}
\newcommand{\B}{\mathcal B}

\newcommand{\se}{\textsf{\tiny E}}
\newcommand{\wC}{\widetilde{\mathcal C}}

\newcommand{\surj}{\twoheadrightarrow}
\newcommand{\wE}{\widetilde\E}
\newcommand{\w}{\widetilde}

\begin{document}

\title{Cohomology in singular blocks for a quantum group at a root of unity}\author{Hankyung Ko}
\address{Max-Planck Institute for Mathematics, Vivatsgasse 7,
53111 Bonn, Germany}
\email{hankyung@math.uni-bonn.de}

\begin{abstract}
Let $U_\zeta$ be a Lusztig quantum enveloping algebra associated to a complex semisimple Lie algebra $\mathfrak g$ and a root of unity $\zeta$. When
$L,L'$ are irreducible $U_\zeta$-modules having regular highest weights, the dimension of $\Ext^n_{U_\zeta}(L,L')$ can be calculated in terms of the coefficients of appropriate Kazhdan-Lusztig polynomials associated to the affine Weyl group of $U_\zeta$. This paper shows for $L,L'$ irreducible modules in a singular block that $\dim\Ext^n_{U_\zeta}(L,L')$ is explicitly determined using the coefficients of parabolic Kazhdan-Lusztig polynomials. This also computes the corresponding cohomology for $q$-Schur algebras and many generalized $q$-Schur algebras. The result depends on a certain parity vanishing property which we obtain from the Kazhdan-Lusztig correspondence and a Koszul grading of Shan-Varagnolo-Vasserot for the corresponding affine Lie algebra.  
\end{abstract}
\maketitle

\section{Introduction}
Let $\mathfrak g$ be a complex semisimple
Lie algebra with root system $R$. Let $\zeta\in\C$ be a primitive $l$-th root of unity for some positive integer $l$.
Let $U_\zeta=U_\zeta(\mathfrak g)$ be the Lusztig's root of unity quantum enveloping algebra (or ``quantum group'') associated
to $\mathfrak g$ over $\C$ introduced in \cite{lus}.
Consider the category $U_\zeta$-mod consisting of integrable finite dimensional $U_\zeta$-modules of type $1$. This is a highest weight category in the sense of \cite{CPShwc} with standard modules $\Delta(\lambda')$, costandard modules $\nabla(\lambda')$, irreducible modules $L(\lambda')$ indexed by their highest weights $\lambda'\in X^+$ where $X^+$ is the set of dominant weights. The affine Weyl group $W_l$ acts on the weight lattice $X$. 
We can write any weight as $w.\lambda$ for $w\in W_l$ and $\lambda\in \overline{C^-}$ where $C^-$ is the standard antidominant $l$-alcove. We view $W_l$ as the Coxeter group ($W_l, S_l)$ where the set of simple reflections $S_l$ consists of the reflections through the walls of $C^-$. Let $I$ be the subset of $S_l$ consisting of the simple reflections that fix $\lambda$. 
The Coxeter system ($W_l,S_l$) also fixes a natural length function $\ell:W_l\to \Z$.
Let $W_I:=\stab_{W_l}(\lambda)$ the subgroup of $W_l$ generated by $I$. Each left coset of $W_{I}$ has a unique minimal element. The set of these minimal coset representatives is denoted by $W^I$. 
If $\lambda'\in W_l.\lambda$, then $\lambda'=w.\lambda$ for a unique $w\in W^I$. By the linkage principle $U_\zeta$-mod decomposes into blocks consisting of the $U_\zeta$-modules whose composition factors have highest weight in the same $W_l$-orbit.

The characters $\ch\Delta(w.\lambda)=\ch\nabla(w.\lambda)$ are given by Weyl's character formula. If $\lambda$ is regular, it is known for most $l$ by works of Kazhdan-Lusztig \cite{KL,KL12,KL3,KL4} and Kashiwara-Tanisaki \cite{KTaffine1,KTaffine2} (see \S\ref{ssrkl} for more details; the condition on $l$ is explained in \S\ref{sskl}) that the characters of irreducibles are given by (an evaluation of) the Kazhdan-Lusztig polynomials, namely,
\begin{equation}\label{inlcf}
\operatorname{ch} L(w.\lambda)= \sum_y (-1)^{\ell( w)-\ell( y)} P_{ y,  w}(-1) \operatorname{ch}\Delta( y. \lambda)
\end{equation}
where the sum is taken over $W^+_l=\{y\in W_l\ |\ y.\lambda\in X^+\}$ (the definition does not depend on $\lambda$ as long as it is regular). See \S \ref{sprem} for this and more notation.

In case the formula \eqref{inlcf} is valid, we further have 
\begin{equation}\label{rhLCF}
\sum_{n=0}^\infty \dim \Ext^n_{U_\zeta}(\Delta(y.\lambda),L(w.\lambda))t^n=t^{\ell( w)-\ell( y)}\bar{P}_{y,w}
\end{equation}
for all $ y, w\in W^+_l$. The bar on the polynomial is the automorphism on $\Z[t,t^{-1}]$ that maps $t$ to $t^{-1}$.

For singular weights (which Lusztig's conjecture does not exclude), one can use the translation functor from a regular orbit to a singular orbit. Applying the translation functor to \eqref{inlcf}, it is immediate that the irreducible character formula for a general dominant weight $w.\lambda$ ($\lambda\in\overline{C^-}$, $w\in W^+_l\cap W^I$) is the alternating sum of the regular character formula:
\begin{align*}\label{singlcf}
\operatorname{ch} L(w.\lambda)&= \sum_y\sum_{x\in W_I} (-1)^{\ell( w)-\ell( yx)} P_{ yx,  w}(-1) \operatorname{ch}\Delta( y. \lambda)
\end{align*}
where the $y$ runs through $W^+_l\cap W^I$.

However, to have an extension formula in singular blocks, the translation is not enough because we cannot determine how to ``sum'' the formula \eqref{rhLCF}. We need a certain parity vanishing property to make it work. This property follows from standard Koszul grading of \cite{SVV}. 

Then the result (Theorem \ref{thm}) is that
\begin{equation}\label{inthLCF}
\sum_{n=0}^\infty \dim \Ext^n_{U_\zeta}(\Delta(y.\lambda),L( w.\lambda))t^n=t^{\ell( w)-\ell( y)}\sum_{x\in W_I} (-1)^{\ell(x)}\bar{P}_{yx,w},
\end{equation}
for $ y, w\in W^+_l\cap W^I$. A similar formula in the finite case was obtained by Soergel \cite{soencohom}, and the formula in our case was conjectured in \cite[Conjecture III]{PS14}. To prove it, we translate the problem into the affine case using the Kazhdan-Lusztig correspondence (\S \ref{sskl}) and use the result of Shan-Varagnolo-Vasserot on affine Lie algebras to get the parity vanishing (\S \ref{sssvv}). Then \eqref{inthLCF} is obtained in \S \ref{main}. 
Before that, we introduce the notion of parity, review some generalities and check necessary properties of translation functors (\S \ref{sprem}). Most of the earlier sections does more than what we need for the proof of our main theorem, tries to depend less on Koszulity, and is developed under an intention of applying them to modular representation theory.

\section*{Acknowledgement}
The author thanks Brian Parshall and Leonard Scott for explaining their conjecture and related subjects to her, pointing out errors in previous proofs, encouraging her to write this into a paper, and carefully reading several versions of this paper. 

\section{Representations for quantum groups at a root of unity}\label{sprem}

Let $U_\zeta$ be a Lusztig quantum group over $\C$ at a primitive $l$-th root of unity \cite[II.H]{J}. Let $\mathcal C^\zeta=U_\zeta$-mod be the category of type $1$ integrable finite dimensional $U_\zeta$-modules. A general theory for the quantum case is developed in \cite{APW}. Though \cite{APW} has restrictions on $l$, it will not be necessary using some later results \cite{AndLink}. So the order $l$ of $\zeta$ can be any positive integer.

We mostly follow the notation of \cite{J} (see the beginnings of \cite[II.1, II.6]{J}). Let $R$ be the root system for $U_\zeta$, $R^+$ be a fixed choice of positive roots, $X$ be the set of (integral) weights, $X^+$ the set of dominant weights, $C$ the bottom dominant $l$-alcove, and $C^-$ the top antidominant $l$-alcove. Then $U_\zeta$-mod is a highest weight category with the poset $X^+=(X^+,\uparrow)$. We denote its standard modules by $\Delta(\lambda)$, costandard modules by $\nabla(\lambda)$, irreducible modules by $L(\lambda)$, and indecomposable tilting modules by $X(\lambda)$ when their highest weight is $\lambda\in X^+$. The projective cover (which is also the injective envelope) of $L(\lambda)$ is denoted by $P(\lambda)$. 

Letting $W$ be the finite Weyl group of $R$, the affine Weyl group $W_l$ is defined as $l\Z R\rtimes W$ and acts on the set $X$ of weights. Let $\rho$ be the sum of all fundamental weights. Equivalently, $\rho$ is the half sum of all positive roots. The action we use is the dot action, that is, $w.\lambda=w(\lambda+\rho)-\rho$ for $w\in W_l, \lambda\in X$. The $W_l$ orbits partition the weights, hence also partition $X^+$. 
This, by the linkage principle, gives a decomposition of the representation category into orbits (we do not call them blocks, because some of the components obtained here are not indecomposable). See \cite[II.6]{J} for a much more detailed discussion. 

Any weight $\lambda'$ (i.e., an element of $X$) is written as $w.\lambda$ for some $w\in W_l$ and a unique $\lambda$ in $\overline{C^-_\Z}=\overline{C^-}\cap X$. We call a weight $\lambda'$ regular if $\lambda \in {C^-}$. We call $\lambda'$ singular if it is not regular. If $\lambda'$ is dominant, the orbit containing $\lambda'$ is represented by this $\lambda\in \overline{C^-}$. 
The choice of $w\in W_l$ is unique if and only if $\lambda$ is regular. If $\lambda$ is regular, this identifies $X^+\cap W_l.\lambda$ with the subset $$W_l^+:=\{w\in W_l\ |\ w.\lambda \in X^+\}$$ of $W_l$. For a general weight $\lambda$, we have preferred representatives. Recall that $W_l$ is generated by the subset $S_l$, which we choose to correspond to the simple reflections through the walls of $C^-$. Let $I:=\{s\in S_l\ |\ s.\lambda=\lambda\}$, $W_I:=\{ w\in W_l\ |\ w.\lambda=\lambda\}$, and let $W^I$ be the set of shortest coset representatives in $W_l/W_I$. Then for $w\in W_l^+$, we have $w\in W^I$ if and only if $w.\lambda$ is in the upper closure of the alcove $w.C^-$. Now define $$W^+(\lambda):=W^{I}\cap W_l^+.$$ We identify $W^+(\lambda)$ with the set of dominant weights in the orbit of $\lambda$. The uparrow ordering of $X^+$ restricted to $W^+(\lambda).\lambda$ agrees with the Coxeter ordering of $W_l$ restricted to $W^+(\lambda)$ \cite[II.8.22]{J}.

We call $\lambda'$ subregular if $\lambda$ is in a codimension one facet in $\overline{C^-}$. Existence of a regular weight is equivalent to $l\geq h$, the Coxeter number. For existence of subregular weights, we have the following elementary fact. 

\begin{prop}\cite[II.6.3]{J}\label{subregexist}
Suppose a regular weight exists, and $l$ is not $30$ if the type is $E_8$; not $12$ if $F_4$; not $6$ if $G_2$. (These are the Coxeter numbers.) Then any wall of $C^-$ contains a weight, that is, for any $s\in S_l$ there exists $\nu\in X$ with $\stab_{W_l}(\nu)=\{e,s\}$. This is the case, in particular, if $l>h$.
\end{prop}

Recall the Coxeter length function $\ell:W_l\to \Z$. By definition the length of the weight $w.\lambda$ is the integer $\ell(\bar w)$ where $\{\bar w\}= W^J\cap wW_J$. We call $w.\lambda$ even if $\ell(\bar w)$ is even, odd if $\ell(\bar w)$ is odd. Also, we say a highest weight module is even if its highest weight is even, odd if its highest weight is odd.

\subsection{Translation functors}\label{sstr}
Fix two weights $\lambda, \mu \in \overline{C^-_\Z}$ and consider the summand $\mathcal C^\zeta_\lambda$ (respectively, $\mathcal C^\zeta_\mu$) of $\mathcal C'=U_\zeta$-mod which consists of the $U_\zeta$-modules whose composition factors are isomorphic to $L(w.\lambda)$ for some $w\in W^+_l$. Denote the translation functor from $\mathcal C^\zeta_\lambda$ to $\mathcal C^\zeta_\mu$ by $\tothewall$. (See, for example, \cite{J} for the algebraic group case. The translation functors in the quantum case are similar and defined in \cite{APW}. Though \cite{APW} assumes that $l$ is an odd prime power, the restriction is unnecessary since we have the linkage principle for all $l$ \cite{AndLink}. See also \cite[\S2.5]{hodge2016remarks}.) Then $\tothewall$ and $\fromthewall$ are biadjoint and both exact. If $\lambda$ and $\mu$ are in the same facet, then $\tothewall$ is an equivalence.

Now assume that $\mu$ is in the closure of the facet containing $\lambda$. \textit{We keep this convention throughout the paper.} Set 
\begin{equation}\label{I,J}
I=\{s\in S_l\ |\ s.\lambda=\lambda\} ,\ \ J=\{s\in S_l\ |\ s.\mu=\mu\}.
\end{equation} Then $W_I=\stab_{W_l}(\lambda)$, $W_J=\stab_{W_l}(\mu)$ are the Coxeter groups generated by $I$ and $J$ respectively. Our convention can now be expressed simply as $I\subset J$.

\begin{prop}\label{tr} Let $y\in W^+(\mu)$. In particular, $y.\mu$ is in the upper closure of the facet containing $y.\lambda$.  
\begin{enumerate}
\item $\tothewall\Delta(yx.\lambda)=\Delta(yx.\mu)=\Delta(y.\mu),$ for any $x\in W_J$.
\smallskip
\item $\fromthewall\Delta(y.\mu)$ has a $\Delta$-filtration whose sections are exactly $\Delta(yx.\lambda)$ where each $x\in W_J/W_I$ occurs with multiplicity one, and we have $$\hd(\fromthewall\Delta(y.\mu))\cong L(y.\lambda).$$
\item $\tothewall L(y.\lambda)=L(y.\mu)$, and $\tothewall L(yx.\lambda)=0$ for any nontrivial element $x\in W_J/W_I$.
\smallskip
\item $[\fromthewall L(y.\mu) : L(y.\lambda)]=|W_J/W_I|$, and we have $$\hd(\fromthewall L(y.\mu))\cong L(y.\lambda), \ \ \soc(\fromthewall L(y.\mu))\cong L(y.\lambda).$$
\item $\tothewall X(yw_J.\lambda)= X(y.\mu)^{\oplus |W_J/W_I|}$, where $w_J$ is the longest element in $W_J$.
\smallskip
\item $\fromthewall X(y.\mu)=X(yw_J.\lambda)$, where $w_J$ is the longest element in $W_J$.
\end{enumerate}
\end{prop}

\begin{proof}
 See \cite[II.7.11, 7.13, 7.15, 7.20]{J} for (1)-(4) and \cite[II.E.11]{J} for (5),(6). They are for algebraic groups and some of them are less general, but all of them are proved in the same way for our setting. 
\end{proof}

\begin{prop}\label{transitivity}
Let $\lambda,\nu,\mu\in\overline{C^-_\Z}$ be such that $\nu$ is contained in the closure of the facet containing $\lambda$, and $\mu$ is contained in the closure of the facet containing $\nu$. Then for any $y\in W^+$ $\fromthe\Delta(y.\mu)\cong T^\lambda_\nu T^\nu_\mu\Delta(y.\mu)$. 
\end{prop}
\begin{proof}
Let $I,J$ as in \eqref{I,J}. We may assume that $y\in W^J$. 

Consider the tilting module $X(yw_J.\lambda)$. We check that both $\fromthe\Delta(y.\mu)$ and $T^\lambda_\nu T^\nu_\mu\Delta(y.\mu)$ are submodules of $X(yw_J.\lambda)$. Since $\Delta(y.\mu)$ is a submodule of the tilting module $X(y.\mu)$, by exactness of translation $\fromthewall\Delta(y.\mu)$ is a submodule of $\fromthewall X(y.\mu)$. But $\fromthewall X(y.\mu)$ is isomorphic to $X(yw_J.\lambda)$ by Proposition \ref{tr} (6). For the same reason $T^\lambda_\nu T^\nu_\mu\Delta(y.\mu)$ is a submodule of $T^\lambda_\nu T^\nu_\mu X(y.\mu)$. The latter is isomorphic to $X(yw_J.\lambda)$, applying Proposition \ref{tr} (6) twice. 

Now note that $\fromthe\Delta(y.\mu)$ and $T^\lambda_\nu T^\nu_\mu\Delta(y.\mu)$ have $\Delta$-filtrations with the same set of sections, i.e, for each $x\in W^I_J=W^I\cap W_J$ the section $\Delta(yx.\lambda)$ appears exactly once. It remains to show that there is only one submodule in $X(yw_J.\lambda)$ which has such a $\Delta$-filtration.

We first determine which standard modules appear in a $\Delta$-filtration of $X(yw_J.\lambda)$. The module $X(y.\mu)$ has a $\Delta$-filtration exactly one of whose sections is isomorphic to $\Delta(y.\mu)$. Any other $\Delta(z.\mu)$ appearing in the filtration satisfies $z<y$. 
Translating to the $\lambda$-block gives the multiplicities of all $\Delta(\lambda')$ in a $\Delta$-filtration of $\fromthewall X(y.\mu)=X(yw_J.\lambda)$ in terms of the $\Delta$-multiplicities of $X(y.\mu)$.
By Proposition \ref{tr}(2), the multiplicity of $\Delta(zx'.\lambda)$, for each $x'\in W_J\cap W^I$, in a $\Delta$-filtration of $X(yw_J.\lambda)$ is the same as the multiplicity of $\Delta(z.\mu)$ in a $\Delta$-filtration of $X(y.\mu)$. Since $\Delta(y.\mu)\not\cong\Delta(z.\mu)$ implies $zW_J\cap yW_J=\emptyset$, we have in that case $\Delta(yx'.\lambda)\not\cong\Delta(zx''.\lambda)$ for all $zx'\in zW_J\neq yW_J \ni zx''$. Therefore, each $\Delta(yx'.\lambda)$ for $x'\in W_J\cap W^I$ appears exactly once in the $\Delta$-filtration of $X(yw_J.\lambda)$.

Suppose $M, M'$ are two submodules of $X(yw_J.\lambda)$ which have $\Delta$-filtrations with the same set of sections $\{\Delta(yx.\lambda)\}_{x\in W^I_J}$. The proposition is proved if we show $M=M'$.

The weight $yw_J.\lambda$ is maximal in $M, M'$ and $X(yw_J.\lambda)$. Also, $yw_J.\lambda$ appears with multiplicity one in all three modules. Hence $M$ and $M'$ contains the unique submodule of $X(yw_J.\lambda)$ isomorphic to $\Delta(yw_J.\lambda)$. Then $M/\Delta(yw_J.\lambda)$ and $M'/\Delta(yw_J.\lambda)$ are submodules of  $X(yw_J.\lambda)/\Delta(yw_J.\lambda)$. Here, each $yw_Js.\lambda$ for $s\in J$ is maximal with multiplicity one. In this way, we can show that $M\cap M'$ has a $\Delta$-filtration with sections $\{\Delta(yx.\lambda)\}_{x\in W^I_J}$. So $M=M'$.
\end{proof}

Composing two opposite translation functors, we get an endofunctor $\fromthewall\tothewall$ on $\mathcal{C}^\zeta_\lambda$. 
In a special case where $\lambda$ is regular and $\mu$ is subregular, the functor $\fromthewall\tothewall$ is commonly called the \textit{$s$-wall crossing functor} and denoted by $\Theta_s$, where $s$ is the unique nontrivial stabilizer of $\mu$. 

Let $\lambda$ be regular, and consider the module $\fromthewall\tothewall\Delta(y.\lambda).$ By Proposition \ref{tr}.(2), there is a filtration $$\fromthewall\tothewall\Delta(y.\lambda)=V_0\supset V_1\supset \cdots \supset V_n=0$$ such that $V_i/V_{i+1}=\Delta(yx_i.\lambda)$. Then $\{x_0=e, \cdots, x_n\}=W_J/W_I$. Since 
\begin{equation}\label{ext1del}
\Ext_{U_\zeta}^1(\Delta(\nu),\Delta(\nu'))=0 \textrm{   for $\nu\not<\nu'$},
\end{equation} we can arrange the filtration in a way that $\ell(x_i)\leq \ell(x_{i+1})$ holds. Now consider the subfiltration $$\fromthe\tothe\Delta(y.\lambda)=U_0 \supset U_1\supset\cdots\supset U_N=0$$ of $\{V_i\}$ where the $i$-th section contains all $\Delta(yx.\lambda)$ with $\ell(x)=i$. Using \eqref{ext1del} again, we have $$U_i/U_{i+1}\cong \bigoplus_{\ell(x)=i, x\in W_J/W_I}\Delta (yx.\lambda).$$

The filtration $\{U_i\}$ is maximal, in some sense, among the filtrations of $\fromthewall\tothewall\Delta(y.\lambda)$ whose sections are direct sums of standard modules. To say in what sense it is so, we prove the following lemma.

\begin{lem}\label{max}
Let $\lambda\in C^-_\Z=C^-\cap X$, $\mu\in\overline{C^-_\Z}$, and $J$ be as in \eqref{I,J}. Then $\Theta_s\Delta(y.\lambda)$, whenever defined, is a subquotient of $\fromthewall\tothewall\Delta(y.\lambda)$ for any $y\in W^+,\ x\in W_J,\ s\in J$.

\end{lem}
We actually state and prove the lemma more generally. The only difficulty it adds is notational. We generalize the $s$-wall crossing functors to define the \textit{facet crossing functor} $\Theta_{J\setminus I}^I:=\fromthe\tothe$ with $I,J$ as in \eqref{I,J}. This is compatible with the wall crossing functor notation as $\Theta_s=\Theta_{\{s\}}^\emptyset$. This notation is useful here because there are many different facets in play. In the other sections we will go back to using $\fromthe\tothe$. Note that the functor $\Theta^I_{J'}$ is defined 
for $J'\subset J\setminus I$ if and only if there exists a weight $\nu$ such that $\{s\in S_l\ |\ s.\nu=\nu\}=I\cup J'$. For the special case $\Theta_s$ in Lemma \ref{max}, this is always the case for $l>h$ by Proposition \ref{subregexist}.

\begin{lem}\label{facetcrossinglem}
Let $\lambda, \mu$, $I\subset J$ as in \eqref{I,J}. For any $J'\subset J\setminus I$, $y\in W^+$, the $J'$-facet crossing module $\Theta_{J'}^I\Delta(y.\lambda)$, whenever defined, is a subquotient of $\Theta_{J\setminus I}^I\Delta(y.\lambda)=\fromthe\tothe\Delta(y.\lambda)$.
\end{lem}

\begin{rem}	\label{A2Uex}
\enumerate
\item A less formal but more illustrative way to state the lemma is that the facet crossings of a standard module are realized in a deeper facet crossing (of the same standard module). 
\item We provide a simple example as another illustration. Let $R$ be type $A$,  $I=\emptyset$ and $J=\{s,t\}\subset S_l$ (i.e., $\lambda$ regular, $\mu$ subsubregular) such that $sts=tst$. Then for any $y\in W^J$, the module $\fromthewall\tothewall\Delta(y.\lambda)$ has six $\Delta$-sections. They are $\Delta(y.\lambda)$, $\Delta(ys.\lambda)$, $\Delta(yt.\lambda)$, $\Delta(yst.\lambda),$ $\Delta(yts.\lambda),$ $\Delta(ysts.\lambda)=\Delta(ytst.\lambda)$. The lemma shows that $\Theta_s\Delta(y.\lambda)=\Theta_s\Delta(ys.\lambda)$, $\Theta_t\Delta(y.\lambda)=\Theta_t\Delta(yt.\lambda)$, $\Theta_s\Delta(yt.\lambda)=\Theta_s\Delta(yts.\lambda)$, $\Theta_t\Delta(ys.\lambda)=\Theta_t\Delta(yst.\lambda)$, $\Theta_s\Delta(yst.\lambda)=\Theta_s\Delta(ysts.\lambda)$, $\Theta_t\Delta(yts.\lambda)=\Theta_t\Delta(ytst.\lambda)$ are realized in $\fromthe\tothe\Delta(y.\lambda)$ as subquotients.
\end{rem}

\begin{proof}[Proof of Lemma \ref{facetcrossinglem}]
Suppose $\Theta^I_{J'}$ is defined, that is, there is a weight $\nu$ such that $\{s\in S_l\ |\ s.\nu=\nu\}=I\cup J'$. Since $\Delta(y.\nu)$ is a subquotient of $T^\nu_\mu\Delta(y.\mu)$, $T^\lambda_\nu\Delta(y.\nu)=\Theta_{J'}\Delta(y.\lambda)$ is a subquotient of $T^\lambda_\nu T^\nu_\mu\Delta(y.\mu)$. But by Proposition \ref{transitivity}, $T^\lambda_\nu T^\nu_\mu\Delta(y.\mu)$ is isomorphic to $\fromthe\Delta(y.\mu)=\Theta_{J\setminus I}\Delta(y.\lambda)$.
\end{proof}
For the rest of the subsection we assume $l>h$ and let $\lambda$, $\mu$, $J$ as in \eqref{I,J} with $\lambda$ regular (that is, $I=\emptyset$). 
\begin{cor}
 Let $y\in W^+(\mu)$. Then $\Theta_J\Delta(y.\lambda)=\fromthe\tothe\Delta(y.\lambda)$ has a filtration each of whose sections is isomorphic to $\Theta_s\Delta(yx.\lambda)$ for some $s\in J$, $x\in W_J$.
\end{cor}
\begin{proof}
By Proposition \ref{subregexist}, for any $s\in J$ the functor $\Theta_s$ is defined on $\mathcal C^\zeta_\lambda$. We can construct a desired filtration using Lemma \ref{max}.
\end{proof}

The following corollary explains the ``maximality'' of the filtration $U_i$.
\begin{cor}\label{maxcor}
We have for all $i$ \begin{equation}\label{maxeq}
\hd U_i= \bigoplus_{\ell(x)=i, x\in W_J} L(yx.\lambda).
\end{equation}
\end{cor}
\begin{proof}
 By construction, the head of $U_i$ contains all $L(yx.\lambda)$ for $\ell(x)=i, x\in W_J$. This shows the ``$\supset$" part. Since the head of any $\Theta_s\Delta(\lambda')$ is irreducible, Lemma \ref{max} shows that it does not contain anything other than those irreducibles. This shows that the inclusion ``$\supset$'' is an equality.
\end{proof}

We know a little more than \eqref{maxeq} about $\{U_i\}$.
\begin{prop}\label{radicalses}
For each $i$, we have 
\begin{enumerate}
\item $U_i\subset \rad^i\fromthe\tothe\Delta(y.\lambda)$ 
\item $\rad U_i =\rad^{i+1}\fromthe\tothe\Delta(y.\lambda)\cap U_i$. 
\end{enumerate}
In other words, the submodule $U_i$ of $U_0=\fromthe\tothe\Delta(y.\lambda)$ has its head in the $i$-th radical layer $\rad^i \fromthe\tothe\Delta(y.\lambda)/\rad^{i+1} \fromthe\tothe\Delta(y.\lambda)$ of $\fromthe\tothe\Delta(y.\lambda)$.
\end{prop}
\begin{proof}
This is clear by Corollary \ref{maxcor} and the fact that the $\Delta$-sections in $\fromthe\tothe\Delta(y.\lambda)$ extends at their heads, that is, \begin{align}\label{deltaext}
\begin{split}
\Ext_{U_\zeta}^1(\Delta(yxs.\lambda),\Delta(yx.\lambda))&\xleftarrow{\cong}\Ext_{U_\zeta}^1(L(yxs.\lambda),\Delta(yx.\lambda))\\
&\xrightarrow{\cong}\Ext_{U_\zeta}^1(L(yxs.\lambda),L(yx.\lambda)),
\end{split}
\end{align} where $s\in J$, $xs<x\in W_J$. Here the first isomorphism is induced by the nonzero map $\Delta(yxs.\lambda)\to L(yxs.\lambda)$ and is a consequence of the Lusztig character formula. See \cite[Theorem 4.3]{cps1}. The second isomorphism is induced by the nonzero map $\Delta(yx.\lambda)\to L(yx.\lambda)$ and is a general fact, which also tells us that the $\Ext$ are one dimensional. See for example \cite[II.7.19 (d)]{J}. Jantzen's proof for $G$-modules works the same for $U_\zeta$-modules.

We provide, nevertheless, a more formal proof. We prove (1), (2) together by induction on $i$. So suppose we have (1), (2) for $i-1$. By Corollary \ref{maxcor}, we have $U_i\subset \rad U_{i-1}$. And induction hypothesis $U_{i-1} \subset \rad^{i}\fromthe\tothe\Delta(y.\lambda)$ implies $\rad U_{i-1} \subset \rad^{i+1}\fromthe\tothe\Delta(y.\lambda)$. Thus (1) holds for $U_i$. The inclusion $\rad U_i\subset \rad^{i+1}\fromthe\tothe\Delta(y.\lambda)\cap U_i$ in (2) now follows from $U_i\subset \rad^{i}\fromthe\tothe\Delta(y.\lambda)$. 

For the other inclusion in (2), suppose for contradiction that $\rad U_i\not\supset \rad^{i+1}\fromthe\tothe\Delta(y.\lambda)\cap U_i$. This means that there is a surjective map $f:U_i\to \Delta(yx.\lambda)$ for some $x\in W_J$
whose restriction to $U_i\cap \rad^{i+1}\fromthe\tothe\Delta(y.\lambda)$ is still surjective. We call the restriction $f'$. Now recall the $\Delta$-filtration $\{V_j\}$ of $\fromthe\tothe\Delta(y.\lambda)$. Take $j$ to be such that $V_j=U_i$. Pick $s\in J$ with $xs<x$. We may assume that (switching the order of the filtration if necessary) there is a short exact sequence $$0\to U_i=V_j\to V_{j-1}\to \Delta(yxs.\lambda)\to 0,$$ and by Lemma \ref{max} there is a surjective map $g: V_{j-1}\to \Theta_s\Delta(yx.\lambda)$ whose restriction to $U_i$ is the map $f$. By \eqref{deltaext}, there is a map $h:\Theta_s\Delta(yx.\lambda)\twoheadrightarrow N$, where $N$ represents a nontrivial element in $\Ext_{U_\zeta}^1(L(yxs.\lambda),L(yx.\lambda))$, and the restriction of $h$ to the submodule $\Delta(yx.\lambda)\subset\Theta_s\Delta(yx.\lambda)$ has image (isomorphic to) $L(yx.\lambda)$. Thus $h\circ g$ is surjective and $h\circ f$, $h\circ f'$ have image $L(yx.\lambda)\subset N$. But this implies that the map $h\circ g$ induces the following two surjective maps
$$V_{j-1}\cap\rad^i\fromthe\tothe\Delta(y.\lambda)\to N/L(yx.\lambda)$$ and $$V_{j-1}\cap\rad^{i+1}\fromthe\tothe\Delta(y.\lambda)\to N/L(yx.\lambda),$$ which is a contradiction. This proves (2) for $U_i$ and completes the induction step.
\end{proof}

\subsection{Affine Kac-Moody Lie algebras and the Kazhdan-Lusztig correspondence}\label{sskl}

We quote Tanisaki's summary in \cite{Tani} of Kazhdan-Lusztig's work and refer the reader to the references therein. Let $\widehat{\mathfrak g}$ be the affine Kac-Moody algebra associated to $R$. And let $\widetilde{\fg}=[\widehat{\fg},\widehat{\fg}]$. Consider $\mathcal O_{k}$, the category $\mathcal O$ for $\widetilde{\mathfrak g}$ at the level $k$. Let $D$ be 
$1$ for type $A_n,D_n,E_n$; $2$ for type $B_n,C_n,F_4$; $3$ for type $G_2$. Let $g$ be the dual Coxeter number, i.e., $g-1$ is the sum of all coefficients of the highest short root. The KL functor $$\mathscr{F}_l:\mathcal O_{-l/2D-g}\to U_\zeta\text{-mod}$$ was defined by Kazhdan and Lusztig in \cite{KL12,KL3,KL4}.
It is often an equivalence of categories. In that case, $\mathscr{F}_l$ maps standard, costandard, irreducible modules to the standard, costandard, irreducible modules of the same index \cite[Theorem 7.1]{Tani}. (It is explained in \cite[\S 6]{Tani} how the modules in $\cO_k$ are indexed by the same highest weights as in $U_\zeta$-mod. The weight $\lambda'$ for $U_\zeta$ is identified with the weight $\lambda'+k\chi$, where $\chi$ is the dual of the central element in $\widetilde{\mathfrak{g}}$. See also \cite{psspecht}. The affine Weyl groups are also different though isomorphic \cite[footnote 11]{psspecht}. )

The rest of the paper usually assumes $\mathscr F_l$ to be an equivalence. The following terminology will be useful.
\begin{defn}
A positive integer $l$ is \textit{KL-good}, for a fixed root system $R$, if the KL functor $\mathscr F_l$ is an equivalence of categories.  
\end{defn}

Some known conditions for $l$ to be KL-good are found 
in \cite{Tani}. For type $A_n$, there is no restriction. For other simply laced cases, $l$ is KL-good if it is greater than or equal to $3$ for $D_n$, $14$ for $E_6$, 20 for $E_7$, and 32 for $E_8$. In non-simply laced cases, also $l$ is KL-good above a bound depending on the type, but they are not known. 


\subsection{Kazhdan-Lusztig theory in regular blocks}\label{ssrkl}
Let $l$ be KL-good for the root system $R$. A consequence of \S \ref{sskl} is the dimension formula for certain coholomogy in a regular block.

Let $P_{x,y}\in \Z[t,t^{-1}]$ be the Kazhdan-Lusztig polynomial defined for each $x,y\in W_l$. They are in fact in $\Z[t^2]$. Take $\lambda\in C^-_\Z$. Then we have
\begin{equation}\label{reghLCF}
\sum_{n=0}^\infty \dim \Ext^n_{U_\zeta}(\Delta(y.\lambda),L(w.\lambda))t^n=t^{\ell( w)-\ell( y)}\bar{P}_{y,w}
\end{equation}
for all $ y, w\in W^+(\lambda)(=W^+_l$, since $\lambda$ is regular). The bar on the polynomial is the automorphism on $\Z[t,t^{-1}]$ that maps $t$ to $t^{-1}$. 

The formula \eqref{reghLCF} follows from the Lusztig character formula by a chain of equivalent conditions \cite[II.C]{J}, independently to the KL-good assumption. Recall that the Lusztig character formula says
\begin{equation}\label{lcf}
\operatorname{ch} L(w.\lambda)= \sum (-1)^{\ell( w)-\ell( y)} P_{ y,  w}(-1) \operatorname{ch}\Delta( y. \lambda),
\end{equation}
 where the sum is over $y\in W^+(\lambda)$. Since the character of $\Delta( y. \lambda)$ is given by the Weyl character formula, this really gives the character of the irreducible. The Lusztig character formula is proved by Kazhdan-Lusztig \cite{KL} 
and Kashiwara-Tanisaki \cite{KTaffine1,KTaffine2} on the affine Lie algebra side (which carries over to the quantum groups by the Kazhdan-Lusztig correspondence), and is extended to all $l>h$ by Andersen-Jantzen-Soergel \cite{AJS} on the quantum side. See \cite[II.H.12]{J} for details and more references.

\section{Grading and parity vanishing}
This section is devoted to proving some lemmas in a more general setting of graded and ungraded highest weight categories and their derived categories. 

\subsection{Parity vanishing}\label{parity}

Let $\D$ be a triangulated category. 

\begin{defn}
Let $\A,\B$ be classes of objects in $\D$. 
\enumerate
\item We say $\A$ is \textit{(left) $\B$-even} (respectively, \textit{$\B$-odd}) if $\Hom^n_\D(X,Y)=0$ for all odd (resp., even) $n$ and all $X\in\A,Y\in\B$. Then $\A$ is said to have \textit{(left) $\B$-parity} if it is either (left) $\B$-even or $\B$-odd. 
\item We say $\A$ is \textit{right $\B$-even} (resp., \textit{$\B$-odd}) if $\Hom^n_\D(Y,X)=0$ for all odd (resp., even) $n$ and all $X\in\A,Y\in\B$. Then $\A$ is said to have \textit{right $\B$-parity} if it is either right $\B$-even or right $\B$-odd.
\end{defn}

Note that $\A$ is $\B$-even if and only if $\B$ is right $\A$-even. In case $\A=\{X\}$, $\B=\{Y\}$, we simply say that $X$ is \textit{$Y$-even} if $\Hom^{n}_\D(X,Y)=0$ for all odd $n$.

\begin{prop} \label{2/3}
Let
$$ X'\to X\to X''\to $$
be a distinguished triangle in $\D$.
If $X'$ and $X''$ are $Y$-even, then $X$ is $Y$-even. If $X'$ and $X''$ are $Y$-odd, then $X$ is $Y$-odd. The same is true for right $Y$-parity. 
\end{prop}
\begin{proof}
This is obvious applying $\Hom( - ,Y)$ and $\Hom(Y, - )$ to the distinguished triangle.
\end{proof}

\begin{defn}
Let $\A$ be a class of objects in $\D$. We define the \textit{even closure} of $\A$ as $$^\se\A:=\{X\in\D\ |\ X \textrm{ is } \A\textrm{-even}\}.$$
Similarly we define the \textit{right even closure} as
$$\A^\se:=\{X\in\D|\ \A \textrm{ is }  X\textrm{-even}\}.$$
For an object $Y\in\D$, we set $^\se Y:=\ ^\se\{Y\}$ and $Y^\se:=\{Y\}^\se$.
\end{defn}
\begin{rem}
We identify a class $\A$ with the full subcategory of $\D$ with objects in $\A$. Though $^\se -$, $-^\se$ are not functors, their images are to be seen as full subcategories. By definition the closures are strict subcategories (i.e., a subcategory such that all objects isomorphic to one of its object belong to it) that contains $0$. By Proposition \ref{2/3}, they are also closed under extension.  
\end{rem}

\begin{prop}\label{basic}
Let $\A\subset\B$ be classes of objects in $\D$. We have
\begin{enumerate}
\item $\A^\se\supset\B^\se$. 
\smallskip
\item $\D^\se=0$, $0^\se=\D$.
\smallskip
\item $(^\se\A)^\se\supset \A$. 
\smallskip
\item $^\se((^\se\A)^\se)=^\se\!\!\A$.
\end{enumerate}
\smallskip

The same relations hold for the right closure.
\end{prop}
\begin{proof}
(1), (2), (3) are immediate from the definition, and (4) follows from (1) and (3).
\end{proof}
It is not true in general $^\se(\A^\se)=(^\se\A)^\se$. An easy example is found when $\D$ a derived category of a highest weight category: Take $\A$ to consist of a single standard object.

\smallskip

The proof of the following proposition is left to the reader.

\begin{prop}\label{adj}
Let $\D$, $\D'$ be triangulated categories and $\A$ be a class of objects in $\D$, $\B$ be a class of objects in $\D'$. Let $L:\D\to\D'$ be a functor and $R:\D'\to\D$ be its right adjoint. Then
\enumerate
\item $(L\A)^\se= R^{-1}(\A^\se)$.
\item $^\se (R\B)= L^{-1}(^\se\B)$.
\end{prop}

\subsection{Parity vanishing in a highest weight category}\label{e}

Let $\mathcal C$ be a highest weight category with an interval finite poset $\Lambda$ of weights. It has standard objects $\Delta(\lambda)$, costandard objects $\nabla(\lambda)$, irreducible objects $L(\lambda)$ for $\lambda\in \Lambda$. We sometimes call the objects in $\mathcal C$ modules. Let us also assume that $\mathcal C$ is over an algebraically closed field so that $\End(L(\lambda))$ is one dimensional for all $\lambda \in\Lambda$. Take the bounded derived category $\mathcal{D}^b(\mathcal{C})$. An object in $\mathcal C$ is identified via the obvious inclusion $\mathcal{C}\to\mathcal{D}^b(\mathcal{C})$ with an object in $\mathcal{D}^b(\mathcal{C})$ concentrated in degree $0$. Note that for $X,Y\in\mathcal C$, we have $\Ext^n_\mathcal{C}(X,Y)=\Hom_{\mathcal{D}^b(\mathcal{C})}(X,Y[n])$. We omit the subscripts and use the notation $\Hom^n(-,-)=\Hom_{\mathcal{D}^b(\mathcal{C})}(-,-[n])$.

We further assume that the set $\Lambda$ is equipped with a length function $l:\Lambda\to\mathbb Z$. Set $\E_0$ to be the full subcategory of $\mathcal{D}^b(\mathcal{C})$ whose objects are the direct sums of $\nabla(\lambda)[\ell(\lambda)+2m]$ for $\lambda\in\Lambda$, $m\in\Z$. Then $\E_i$ is defined inductively as the full subcategory of $\mathcal{D}^b(\mathcal{C})$ such that 
$$X\in\E_i \Leftrightarrow \text{there is a distinguished triangle   } X'\to X\to X''\to \text{ with } X'\in\E_{i-1}, X''\in \E_0 .$$

Set $\E$ to be the union $\bigcup_i\E_i$. This is by construction a subcategory of $\E^R$ defined in \cite{cps1}, whose defining condition is 
$$X\in\E_i \Leftrightarrow \text{there is a distinguished triangle   } X'\to X\to X''\to \text{ with } X', X''\in \E^R_{i-1} ,$$ with $\E_0=\E^R_0$. 
In fact, it is implicit in (the proof of) the recognition theorem \cite[(2.4) Theorem]{cps1} that $\E^R=\E$. We make it explicit.

\begin{prop}\label{recog}
Let $\A$ be a class of objects in $\mathcal{D}^b(\mathcal{C})$. Then the following conditions are equivalent.\enumerate
 \item $\A\subset \E^R$.
 \item $\A\subset \E$.
\item For each $X\in \A$, we have $\Hom^n(\Delta(\lambda),X)=0$ for all $\lambda\in\Lambda$ and all integers $n\not\equiv \ell(\lambda)$ mod 2. 
\end{prop}
\begin{proof}
It is enough to consider the case in which $\A$ consists of a single object $X$. The implications (2) $\Rightarrow$ (1) $\Rightarrow$ (3) are clear. (3) $\Rightarrow$ (2) is the only nontrivial step. Although it is proved in the proof of \cite[(2.4) Theorem]{cps1}, we provide a full proof because it contains an important construction. 

Suppose $\Hom^n(\Delta(\lambda),X)=0$ for $n\not\equiv \ell(\lambda)$ mod $2$. Let $Y_0= X$. We show that we can construct $Y_0,\cdots,Y_i\in\mathcal{D}^b(\mathcal{C})$ inductively. It is enough to show that we can find a distinguished triangle $ Y_{i+1}\to Y_{i}\to\nabla(\lambda_i)[n_i]\to$ such that (i) $n_i\equiv \ell(\lambda_i)$ mod $2$; (ii) the cohomology $H^\bullet(Y_{i+1})$ has composition factors with lower highest weights compared to the composition factors in $H^\bullet(Y_i)$ (the meaning of this condition will become clearer in the course of the proof); (iii) $\Hom^n(\Delta(\lambda),Y_{i+1})=0$ for $n\equiv \ell(\lambda)+1$ mod 2. Pick a maximal weight $\lambda_i$ among the highest weights of the composition factors in $H^\bullet(Y_i)$. Say it is in $H^{n_i}(Y_i)$. Since $\lambda_i$ is maximal, by universal property of $\nabla(\lambda_i)$, there is a nonzero map from $H^{n_i}(Y_i)$ to $\nabla(\lambda_i)$.
This map lifts to a morphism from $Y_i$ to $\nabla(\lambda_i)[n_i]$ in the derived category ${\mathcal{D}^b(\mathcal{C})}$. 
So we get a distinguished triangle $ Y_{i+1}\to Y_{i}\to\nabla(\lambda_i)[n_i]\to$. 
Since we took a map to $\nabla(\lambda_i)$ whose preimage contains a composition factor of $H^\bullet (Y_i)$ isomorphic to $L(\lambda_i)$, we have
$$[H^\bullet (Y_{i+1}):L(\lambda_i)]<[H^\bullet(Y_i):L(\lambda_i)],$$
and all the other differences between $H^\bullet (Y_i)$ and $H^\bullet (Y_{i+1})$ involve only the composition factors in $\nabla(\lambda_i)/ L(\lambda_i)$ which only has weights lower than $\lambda_i$. Thus we have the condition (ii). 
Since $\Hom^n(\Delta(\lambda_i),Y_{i})=0$ for $n\equiv \ell(\lambda_i)+1$ mod 2, the $n_i$ should satisfy the condition (i). 
Finally (i) and the right $\Delta(\lambda)$-parity of $Y_i$ implies (iii).
\end{proof}

\begin{rem} \label{s.0 example}
\leavevmode
\begin{enumerate}
\item In fact, the construction of the distinguished triangle in the proof does not use the right $\Delta(\lambda)$-parity of $X$. The same induction in the proof works removing the conditions (i), (iii). This shows that all complexes are filtered by shifts of costandard modules. A complex belongs to the category $\E$ when there appear the ``correct shifts" only. For example, let $\mathcal C$ be (a truncation of) $U_\zeta$-mod with $l\geq h$. So $0$ is a regular weight, and $L(0)=\Delta(0)=\nabla(0)$. Denoting by $s$ the reflection through the upper wall of $C$, we have short exact sequences $0\to L(0)\to\Delta(s.0)\to L(s.0)\to 0$ and $0\to L(s.0)\to \nabla(s.0)\to  L(0) \to 0$ of $U_\zeta$-modules in the orbit of the weight $0$. Then $\Delta(s.0)$ is not in $\E^R$, even up to shifts, because both $\nabla(0)$ and $\nabla(0)[-1]$ appear when one applies the above construction of distinguished triangles:
$$\nabla(0)\oplus \nabla(0)[-1]=L(0)\oplus L(0)[-1]\cong Y_1\to Y_0=\Delta(s.0)\to \nabla(s.0)\to,$$
$$\nabla(0)[-1]\cong Y_2\to Y_1\to \nabla(0)  \to,$$
$$0= Y_3\to Y_2\to \nabla(0)[-1]  \to.$$

\item If the $Y_i, \lambda_i, n_i$ are as in the proof, the character of $X$ is given by $\Sigma_i (-1)^{n_i}[\nabla(\lambda_i)].$ By (1) this is true for any $X\in \mathcal{D}^b(\mathcal{C})$. Then $X$ is in $\E^R$ if and only if there is no cancellation in the character formula. In the example above, $\nabla(0)$ and $\nabla(0)[-1]$ cancel each other in characters, hence are invisible in the character formula.
\end{enumerate}
\end{rem}
\medskip


We are mostly interested in the case in which $\A$ in Proposition \ref{recog}  is the set $\{L(\lambda)[\ell(\lambda)]\ |\ \lambda\in\Lambda\}$. We say that $M\in \mathcal C$ has parity if it has $L$-parity for any irreducible $L\in \mathcal C$. This is equivalent to $M$ having a parity projective resolution, i.e., a projective resolution $P_\bullet$ such that $$P(\lambda) \textrm{ is a direct summand of } P_i\Rightarrow i\equiv \ell(\lambda)+\epsilon \textrm{ mod 2},$$ where $\epsilon$ is either $0$ or $1$ (uniformly). Then $\{L(\lambda)[\ell(\lambda)]\ |\ \lambda\in\Lambda\}\subset \E^L\cap \E^R$ if and only if all standard modules have parity. (The $\epsilon$ in a parity projective resolution of $\Delta(\lambda)$ is determined by the equality $\epsilon\equiv \ell(\lambda)$ mod 2.) Following \cite{cps1}, we say $\mathcal C$ has a \textit{Kazhdan-Lusztig theory} if the set $\{L(\lambda)[\ell(\lambda)]\ |\ \lambda\in\Lambda\}$ is contained in $\E^R$ (and $\E^L$, but the two conditions are the same under duality).

In the case of $U_\zeta$-modules, each $L(w.\lambda)[\ell(w)]$ for $\lambda\in C^-_\Z$ does belong to $\E^L\cap \E^R$. (The length function we use in defining $\E^L$ and $\E^R$ is, of course, the usual length function on $W_l$.) This follows from Proposition \ref{recog} and \eqref{reghLCF} (and its dual), since $P_{x,y}$ is a polynomial on $t^2$.

Letting $\D=\mathcal{D}^b(\mathcal{C})$, the recognition theorem can be formulated in our notation from \S \ref{parity} as follows.
\begin{prop}\label{pvcat}
We have 
$$(\E^L_0)^\se=\E^R \textrm{ and }^\se(\E_0^R)=\E^L.$$
\end{prop}

An immediate consequence of this (and Proposition \ref{basic}) is that $\E^R,\E^L$ are \textit{closed} in the sense that $(^\se(\E^R))^\se=\E^R$ and $^\se((\E^L)^\se)=\E^L$.

\begin{ex}
Consider $\mathcal C^\zeta=U_\zeta$-mod. Let $F$ be a facet in $\overline{C^-}$ and $\lambda\in F\cap X$. Suppose $\mu$ is a weight in $\overline F\setminus F$. Let $M\in \mathcal C^\zeta_\mu$. If $\fromthe M\in \E^R$ (defined in \S \ref{e}), then $M=0$.

This is proved as follows. By Proposition \ref{adj} and Proposition \ref{pvcat} below, we have $$(\tothe\E^L_0)^\se=(\fromthe)^{-1}((\E^L_0)^\se)=(\fromthe)^{-1}\E^R.$$ But since $$\tothe\Delta(y.\lambda)[\ell(y)+2m]=\Delta(y.\mu)[\ell(y)+2m],$$ $$\tothe\Delta(ys.\lambda)[\ell(y)+1+2m]=\Delta(y.\mu)[\ell(y)+1+2m]$$ for $y\in W^J$, $s\in J\setminus I$, $m\in\Z$, all shifts of $\Delta(y.\mu)$ for all (dominant) $y.\mu$ belong to $\tothe\E^L_0$. So if $\fromthe M\in \E^R$, then $\Hom^n(\Delta(y.\mu),M)=0$ for all $n$, hence $M=0$.
\end{ex}

\subsection{Linearity and parity}\label{lin}
In this section, we consider positively graded highest weight categories. Let $\mathcal C$ be a highest weight category as in \S \ref{e}. Identify $\mathcal C$ with the category of (finite dimensional) $A$-modules for some (finite dimensional quasi-hereditary) algebra $A$. What we assume now is that $A$ is a positively graded algebra and $A_0$ is semisimple. We let $\wC$ be the category of graded $A$-modules. So we have the ``forget the grading'' functor $F:\wC\to \mathcal C$ with $F\langle 1\rangle\cong F$. Here $\langle 1\rangle$ is the grade shift defined by $(M\langle 1\rangle)^i=M^{i-1}$ where $M^i$ denotes the grade $i$ component of $M\in\wC$. 

We call a graded module $\widetilde M\in \wC$ a \textit{(graded) lift} of $M\in \mathcal C$ if $F(\widetilde M)\cong M$. For any irreducible $L(\lambda)\in \mathcal C$, let $\widetilde L(\lambda)\in \wC$ be the irreducible of highest weight $\lambda$ concentrated in grade $0$, let $\widetilde \Delta(\lambda)$ be the lift of $\Delta(\lambda)$ whose head is $\widetilde L(\lambda)$, let $\widetilde\nabla(\lambda)$ be the lift of $\nabla(\lambda)$ whose socle is $\widetilde L(\lambda)$, let $\widetilde P(\lambda)$ be the projective cover of $\widetilde L(\lambda)$ in $\wC$, and let $\widetilde I(\lambda)$ be the injective envelope of $\widetilde L(\lambda)$ in $\wC$. Of course, $\w P(\lambda)$ lifts $P(\lambda)$ and $\w I(\lambda)$ lifts $I(\lambda)$.

Recall that $M\in\widetilde{\mathcal C}$ is called \textit{linear} if it has a projective resolution $P=P_\bullet$ such that the head of $P_{-i}$ is homogeneous of grade $i$, in other words, $\ext^n(M,\widetilde L(\lambda)\langle i\rangle)=0 \textrm{   unless $i=n$}$ for any $\lambda\in\Lambda$. We call such a projective resolution a \textit{linear projective resolution}. By definition, $\wC$ is \textit{Koszul} if each irreducible $\widetilde L(\lambda)$ is linear for any $\lambda\in\Lambda$. It is \textit{standard Koszul} if each standard module $\widetilde\Delta(\lambda)$ for $\lambda\in\Lambda$ is linear and each costandard module is \textit{colinear}, i.e., has an injective resolution $I_\bullet$ such that the socle of $I_i$ is homogeneous of grade $-i$. If $\mathcal C$ has a duality, then the condition on costandard modules follows from the one on standard modules.

Compare the following with Proposition \ref{2/3}. 
\begin{prop}\label{linlem} 
Suppose there is a short exact sequence $$0\to M\to M'\to M''\to 0$$ in $\widetilde{\mathcal{C}}$. Suppose $M',M''$ are linear. If $M$ is concentrated in grades $\geq 1$, then $M\langle -1\rangle$ is linear.
\end{prop}
\begin{proof}
 Let $P, P', P''\in \mathcal D^b(\wC)$ be minimal projective resolutions of $M, M', M''$ respectively. Automatically, $P',P''$ are linear. 
 There is a distinguished triangle
$$P\to P'\to P''\to P[1]\to.$$ Positivity of grading and the assumption on $M$ implies that the degree $n$ term of $P_n$ of $P$ is generated by grade $n+1$ or greater. By linearity, the kernel of $P'\to P''$ in degree $n$ should be generated by grade $n$, but the image of $P\to P'$ is in grades $n+1$ or greater. This shows that the map $P\to P'$ is zero (in each degree).
So we have a short exact sequence
$$0\to P'\to P''\to P[1]\to 0.$$ It follows that $P[1]$ is linear, and so is $P\langle -1\rangle=P[1][-1]\langle -1\rangle$. Hence $M\langle -1\rangle$ is linear. 
\end{proof}

\begin{cor}\label{nonlin}
Suppose there is a short exact sequence $$0\to M\to M'\to M''\to 0$$ in $\widetilde{\mathcal{C}}$, and $M',M''$ linear. If $M$ is concentrated in grades $\geq 2$, then $M$ is $0$.
\end{cor}
\begin{proof}
By Proposition \ref{linlem}, there is a surjective map $P_0\to M$ where $P_0\in\wC$ is generated by its components in grade 1. Since $M$ is concentrated in grades $\geq 2$, the image of the map $P_0\to M$ is zero, and hence $M=0$.
\end{proof}

There are analogues of the categories $\E^R, \E^L$ for $\mathcal D^b(\wC)$. The category $\widetilde\E^R$ (denoted by $\E^R$ in \cite{PS14}) is defined as the union of $\wE^R_i$ where $\wE^R_i$ is defined inductively as follows. Set $\wE^R_0$ to be the full subcategory of $\mathcal{D}^b(\wC)$ whose objects are the direct sums of $\widetilde\nabla(\lambda)\{m\}$ for $\lambda\in\Lambda$, $m\in\Z$. Here $\{-\}$ is the shift defined as $\{1\}=\langle 1\rangle[1]$. Then we define $\wE^R_i$ to be the full subcategory of $\mathcal{D}^b(\wC)$ such that $$X\in\wE^R_i \Leftrightarrow \text{there is a distinguished triangle   } X'\to X\to X''\to \text{ with } X'\in\wE^R_{i-1}, X''\in \wE^R_0 .$$ The dual category $\wE^L$ is defined dually.
There is also a version of the recognition theorem (Proposition \ref{recog}), which is proved in a similar way.

\begin{prop}\cite[Theorem 3.3]{PS14}
Let $X\in\mathcal{D}^b(\wC)$. Then \\
$$X\in \wE^R \Leftrightarrow \textrm{ $\Hom_{\mathcal{D}^b(\wC)}^n(\widetilde\Delta(\lambda),X\langle m\rangle)\neq 0$ implies $m=n$ (for all $\lambda\in \Lambda$).}$$
\end{prop}

Thus, standard Koszulity (and its dual) is equivalent to that $\wE^R$ (and $\wE^L$) contains all irreducibles in $\mathcal D^b(\wC)$. We can combine $\wE^R$ and $\E^R$ to define a category studied in \cite{cpshodual}. We will call it $\E_{gr}^R$, following \cite[\S 1.3]{cpshodual}. Let $\E_{\textrm{gr},0}^R:=\wE_0 \cap \E^R_0$, where we view $\E^R_0$ a subcategory of $\mathcal D^b(\wC)$, pulling back via the forgetful functor. Thus $\E_{\textrm{gr},0}^R$ consists of direct sums of $\widetilde\nabla(\lambda)\{\ell(\lambda)+2m\}$, $m\in\Z$, $\lambda\in\Lambda$. The category $\E_{\textrm{gr}}^R$ is the union of all $\E_{\textrm{gr},i}^R$, where $\E_{\textrm{gr},i}^R$ is inductively defined as $$X\in\E^R_{\textrm{gr},i} \Leftrightarrow \text{there is a distinguished triangle   } X'\to X\to X''\to \text{ with } X'\in\E^R_{\textrm{gr},i-1}, X''\in \E^R_{\textrm{gr},0}.$$ Using this, the notion of a graded Kazhdan-Lusztig theory is introduced in \cite{cpshodual}: $\mathcal{C}$ is said to have a \textit{graded Kazhdan-Lusztig theory} if $\E^R_{\textrm{gr}}$ contains $\{L(\lambda)\{\ell(\lambda)+2m\}\ |\ \lambda\in\Lambda, m\in\Z\}$.

We have the third recognition theorem.
\begin{prop}\cite[Theorem 1.3.1]{cpshodual}
Let $X\in\mathcal{D}^b(\wC)$. Then \\
$$X\in \E_{\textrm{gr}}^R \Leftrightarrow \textrm{ $\Hom_{\mathcal{D}^b(\wC)}^n(\widetilde\Delta(\lambda),X\langle m\rangle)\neq 0$ implies $m=n$ and $n\equiv \ell(\lambda)$ (for all $\lambda\in \Lambda$).}$$
\end{prop}
This shows that $\E_{\textrm{gr}}^R=F^{-1}\E^R\cap \wE^R$, where $F$ is the forgetful functor from $\mathcal D^b(\wC)$ to $\mathcal D^b(\mathcal C)$ induced by the forgetful functor from $\wC$ to $ \mathcal C$. Therefore, $\mathcal C$ has a graded Kazhdan-Lusztig theory if and only if $\mathcal C$ has a Kazhdan-Lusztig theory and is standard Koszul.
\smallskip

We conclude the section by presenting a relation between linearity and parity. It will apply to our case. 
\begin{prop}\label{linearpairty}
Suppose we have $\Ext^1_{\mathcal C}(L(\lambda_1),L(\lambda_2))=0$ whenever $\ell(\lambda_1)\equiv \ell(\lambda_2)$ mod 2. If $M\in \mathcal C$ has a linear lift $\widetilde M\in \wC$, then $M$ has parity. In particular, standard Koszulity implies a Kazhdan-Lusztig theory.
\end{prop}
\begin{proof}
Let $P_\bullet$ be a linear projective resolution of $\widetilde M$. Then $P_i\to P_{i+1}$ maps the head of $P_i$, which is in grade $-i$, to the second radical layer of $P_{i+1}$. Then by Lemma \ref{lem} below, $P_i$ and $P_{i+1}$ have opposite parity. In other word, $P_\bullet$ is a parity resolution of $\widetilde M$. Let $L$ be any irreducible object in $\mathcal C$. Then $\Ext_\mathcal{C}^n(M,L)=\Hom_\mathcal{C}(P_{-n},L)$ can be nonzero only when $P_{-n}$ and $L$ have the same parity, thus $M$ has $L$-parity. The claim follows.
 The last sentence of the Proposition is obtained by taking $M$ to be a standard module.
\end{proof}

\section{Koszulity and singular Kazhdan-Lusztig theory}\label{koszul}

Let for $J\subset S_l$ and $y,w\in W^J$
$$P^J_{y,w}:=\sum_{x\in W_J}(-1)^{\ell(x)}P_{yx,w}.$$ 
This is called a \textit{parabolic Kazhdan-Lusztig polynomial} \cite{Deodpara,KTparabolic}.

Our goal is to show that
\begin{equation}\label{hLCF}
\sum_{n=0}^\infty \dim \Ext^n_{U_\zeta}(\Delta(y.\mu),L( w.\mu))t^n=t^{\ell( w)-\ell( y)}\bar{P}^{J}_{y,w}
\end{equation}
holds for all $\mu\in \overline{C^-_\Z}$, $y,w\in W^+(\mu)$, where $J=\{s\in S_l\ |\ s.\mu=\mu\}$. Assuming that $l$ is KL-good, it is enough to prove the formula \eqref{hLCF} in $\cO$ at the negative level $k=-l/2D-g$. We use the same notation for standard, costandard, irreducible objects in $\cO_k$ as in $U_\zeta$-mod. Any $\mu\in\overline{C^-_\Z}$, which is a weight for $U_\zeta$, determines $\widetilde\mu=\mu+k\chi$, a weight for $\cO_k$. So the Kazhdan-Lusztig correspondence maps $\Delta(\widetilde{w.\mu})$ to $\Delta(w.\mu)$, $L(\widetilde{w.\mu})$ to $L(w.\mu)$, etc. We have $w.\w\mu=\w{w.\mu}$ if we identify the affine Weyl group for $U_\zeta$ with the one for $\w \fg$ as in \S \ref{sskl}.

To apply \cite{SVV} more easily, we further identify the extension spaces to the ones in $\mathbb{O}$, the category $\cO$ for $\widehat{\mathfrak{g}}$. 
Given a weight $\widetilde\mu=\mu+k\chi$ for $\widetilde{\fg}$, we fix a weight $$\widehat{\mu}:=\mu+k\chi +b\delta$$ for $\widehat{\fg}$, where $\delta$ is the fundamental imaginary root and $b$ is some number we don't care as long as it makes $\widehat{\mu}$ lie out of the critical hyperplanes. By \cite[Corollary 3.2]{psspecht} and the preceding footnote, the orbit of $\widetilde\mu$ in $\cO$ is isomorphic to the orbit of $\widehat{\mu}$ in $\mathbb{O}^+$. Here $\mathbb{O}^+$ is the full subcategory of $\mathbb{O}$ consisting of the modules whose composition factors are of integral dominant highest weight (dominant for the subalgebra $\fg$). 
Recall that the integral Weyl group of $\widehat{\mu}$ is defined to be generated by the simple reflections corresponding to the simple roots $\alpha$ such that $ (\alpha,\alpha) \textrm{ divides }2(\widehat{\mu} + \rho, \alpha)$, where $(-,-)$ is the usual bilinear form on $\widehat{\mathfrak{h}}^*$. 
It is isomorphic to $W_l$ as a Coxeter group, since $\widehat{\mu}$ lies out of the critical hyperplanes. 
We denote this integral Weyl group by $W_l$ for convenience and keep the notation in \S \ref{sprem}. 
We can also keep the Coxeter ordering on $W^+(\mu)$ as the poset ordering \cite[Appendix I]{psspecht}. 
In this setting, the formula \eqref{hLCF} is equivalent to 
\begin{equation}\label{affhLCF}
\sum_{n=0}^\infty \dim \Ext^n_{\mathbb{O}^+}(\Delta(y.\widehat\mu),L( w.\widehat\mu))t^n=t^{\ell( w)-\ell( y)}\bar{P}^{J}_{y,w}
\end{equation}
 for $\mu\in \overline{C^-_\Z}$, $y,w\in W^+(\mu)$. 

Given a highest weight category $\mathcal C'$ with poset $\Lambda$, a truncation $\mathcal{C}=\mathcal C'[\Gamma]$ by a poset ideal $\Gamma\subset \Lambda$ is defined to be the Serre subcategory of $\mathcal C'$ generated by $\{L(\gamma)\ |\ \gamma\in\Gamma\}$. Its objects are those with composition factors of the form $L(\gamma)$, $\gamma\in\Gamma$. The category $\mathcal C$ satisfies 
\begin{equation}\label{recoll}
\Ext_\mathcal{C}^n(X,Y)=\Ext_\mathcal{C'}^n(X,Y)
\end{equation}
for $X,Y\in\mathcal C$ by the general theory of highest weight categories \cite[Theorem 3.9]{CPShwc}. Applying this to the case $\mathcal C'=\mathbb O^+$, it is enough to prove \eqref{affhLCF} in $\mathcal{C}=\mathbb O^+[\Gamma]$ for a finite ideal $\Gamma$ containing $y.\widehat{\mu},w.\widehat{\mu}$. 

\subsection{Koszul grading and parity vanishing}\label{sssvv}
We assume in this subsection that the level $k$ is an integer. This is in order to use the result of \cite{SVV}. We also assume that $l>h$. We will see later that these restrictions are not necessary for our result.

Let $\mathcal C_{\widehat\lambda}$, $\mathcal C_{\widehat\mu}$ be truncations of $\widehat\lambda$ and $\widehat\mu$ orbits as in \cite[\S 3.4]{SVV}. That is, there is some $v\in W_l$, which we do not keep track of, such that $\mathcal C_{\widehat\lambda}=\mathbb O^+[\Lambda]$ where $\Lambda=\{w.\widehat\lambda\in W^+.\widehat\lambda\ |\ w\leq v \}$. And $\mathcal C_{\widehat\mu}$ is similarly defined. (In the notation of \cite{SVV}, $\mathcal C_{\widehat\lambda}=^v\!\!\mathbf O^\emptyset_{I,-}$ and $\mathcal C_{\widehat\mu}=^v\!\!\mathbf O^\emptyset_{J,-}$.) We assume that $\widehat\lambda$ is regular.\footnote{We need neither fix the level $k$ nor assume $l>h$, as the translation functors can move the level. But we make this assumption anyway, because it is easy to take care of the restriction on $k$ altogether when we treat the case of non-integer $k$. See the proof of Theorem \ref{thm}.} Then we have the following.

\begin{thm}\cite[Theorem 3.12, Lemma 5.10]{SVV}\label{svv}
The categories $\mathcal C_{\widehat\lambda}$, $\mathcal C_{\widehat\mu}$ are standard Koszul. Letting $\widetilde{\mathcal C}_{\widehat\lambda}$, $\widetilde{\mathcal C}_{\widehat\mu}$ be the corresponding categories of graded modules, there is a graded translation functor $\widetilde\tothe:\widetilde{\mathcal C}_{\widehat\lambda}\to\widetilde{\mathcal C}_{\widehat\mu}$ which lifts the (ungraded) translation functor $\tothe:\mathcal C_{\widehat\lambda}\to \mathcal C_{\widehat\mu}$. (See \cite[Proposition 4.36]{SVV} and the remark below.) That is, $F\circ \widetilde\tothe\cong \tothe \circ F$ where $F$ is the functor (on both $\widetilde{\mathcal C}_{\widehat\lambda}$ and $\widetilde{\mathcal C}_{\widehat\mu}$) that forgets the grading. The functor $\widetilde\tothe$ satisfies $\widetilde\tothe\widetilde L(w.\widehat{\lambda})=\widetilde L(w.\widehat{\mu})$ for $w\in W^J$. 
\end{thm}
\begin{rem}
The condition ``$d+N>f$'' in \cite[Lemma 5.10]{SVV} or a similar condition in \cite[Proposition 4.36]{SVV} says that the difference between the level of $\widehat{\mu}$ and the level of $\widehat\lambda$ is less than the dual Coxeter number $g$. (The dual Coxeter number is denoted by $N$ in \cite{SVV}. The numbers $d,f$ in \cite{SVV} are such that $-d-N$ and $-f-N$ are the levels of the weights.) But to use the translation in \cite{KTtranslation}, as the beginning of the proof of \cite[Proposition 4.36]{SVV} does, a different assumption on the weights is required: Given two integral affine weights $\nu,\xi$ of (not necessarily the same) negative levels, the translation $T_\nu^\xi$ from the orbit of $\nu$ in $\mathbb O$ to the orbit of $\xi$ in $\mathbb O$ as in \cite[\S 3]{KTtranslation} exists if $\xi-\nu\in W_aP^+$ where $P^+$ is the set of integral dominant (affine) weights for $\widehat{\fg}$ and $W_{a}$ is the (affine) Weyl group of $\widehat{\fg}$ (See also \cite[\S 2]{KTtranslation}). This assumption is different from and not implied by the condition $d+N>f$.

 We can instead construct the desired translation in two steps as follow. As in \cite{SVV}, it is enough to define a translation $T_\nu^\xi$ where $\nu$ is a regular (integral) weight. Then $T_\xi^\nu$ can be defined to be its left adjoint. Let $\widehat\rho:=\rho+g\chi$ be our choice of an ``affine $\rho$''. Now the antidominant alcove in this setting can be defined by the condition $\langle \xi+\widehat{\rho},\alpha \rangle< 0$ for all affine root $\alpha$. Then, given any antidominant integral weight $\xi$, the weights $\xi+n\widehat\rho$, $\nu+n\widehat{\rho}$ are integral for each $n\in\Z$. They are dominant if $n$ is sufficiently large. Take such an $n$ which is also positive. Now $\xi-(-n\widehat\rho), \nu-(-n\widehat{\rho})\in P^+\subset W_aP^+$ defines the translations $T_{-n\widehat\rho}^\xi$ and $T_{-n\widehat{\rho}}^\nu$.
Note that $\nu$ and $-n\widehat\rho$ are in the same facet, the antidominant alcove. This implies the translation functor $T_{-n\widehat{\rho}}^\nu$ is an equivalence (See for example \cite[Propositions 3.6, 3.8]{KTtranslation} and the comparison theorem \cite[Theorem 5.8]{psottawa}, or see \cite[\S 6]{psspecht}). We fix an inverse and call it $T_\nu^{-n\widehat\rho}$. Since $T_\nu^{-n\widehat\rho}$ is an inverse of a translation functor, it behaves just like a classical translation functor. Finally, let $T_\nu^\xi:=T_{-n\widehat\rho}^\xi\circ T_\nu^{-n\widehat\rho}$. The functor $T_\nu^\xi$ has all the properties that the classical translations have. Therefore, the rest of \cite[Proposition 4.36, Lemma 5.10]{SVV} works.  
\end{rem}

\medskip


Let $$\widetilde\fromthe:\widetilde{\mathcal C_{\widehat\mu}}\to\widetilde{\mathcal C_{\widehat\lambda}}$$ be a left adjoint of $\widetilde\tothe$. Its existence follows from the adjoint functor theorem because we are dealing with finite number of irreducible objects and $\End(L)=\C$ for each irreducible $L$. 

We want the translation functors in Theorem \ref{svv} for $\widehat{\fg}$ (restricted to $\mathcal O_k$) to agree with the translation functors for $U_\zeta$-mod via the KL correspondence. To avoid discussing this problem, we redefine the translation $\tothe :\mathcal C^\zeta_\lambda\to\mathcal C^\zeta_\mu$ to be $\mathscr F_l(\tothe)$ and $\fromthe :\mathcal C^\zeta_\mu\to\mathcal C^\zeta_\lambda$ to be $\mathscr F_l(\fromthe)$. Then everything we need from \S \ref{sstr} is still true by the same proof using the basic properties in \cite[Proposition 4.36]{SVV}. We denote $\Ext^n_{\widetilde{\mathcal C}}(-,-)$ by $\ext^n_{\mathcal C}(-,-)$ and $\Hom_{\widetilde{\mathcal C}}(-,-)$ by $\hom_{\mathcal C}(-,-)$.

\begin{cor}\label{crossinglinear}
The module $\widetilde\fromthe\widetilde\tothe\widetilde\Delta(y.\widehat\lambda)$ is linear for any $y\in W^J$.
\end{cor}
\begin{proof}
 Adjunction gives for all $n, i$
\begin{align*}
\ext^n_{\mathcal C_{\widehat{\lambda}}}(\widetilde\fromthe\widetilde\tothe\widetilde\Delta(y.\widehat\lambda),\widetilde L(w.\widehat\lambda)\langle i\rangle)&\cong\ext^n_{\mathcal C_{\widehat{\lambda}}}(\widetilde\tothe\widetilde\Delta(y.\widehat\lambda),\widetilde\tothe\widetilde L(w.\widehat\lambda)\langle i\rangle)\\
&\cong\ext^n_{\mathcal C_{\widehat\mu}}(\widetilde\Delta(y.\widehat\mu),\widetilde L(w.\widehat\mu)\langle i\rangle),
\end{align*} 
which is $0$ unless $n= i$ by standard Koszulity of $\mathcal C_{\widehat\mu}$. Thus $\widetilde\fromthe\widetilde\tothe\widetilde\Delta(y.\widehat\lambda)$ is linear.
\end{proof}
\begin{rem}
In fact, a linear projective resolution of $\widetilde\fromthe\widetilde\tothe\widetilde\Delta(y.\widehat\lambda)=\widetilde\fromthe\widetilde\Delta(y.\widehat\mu)$ is obtained by applying the translation to a linear projective resolution of $\widetilde\Delta(y.\widehat\mu)$. Let $P_\bullet$ be one. It is obvious that $\widetilde\fromthe P_\bullet$ is a projective resolution of $\widetilde\fromthe\widetilde\Delta(y.\widehat\mu)$. For linearity, we check $$\widetilde\fromthe \widetilde P(w.\widehat\mu)\cong\widetilde P(w.\widehat\lambda).$$ This is true up to grading shift by \cite[II.7.16]{J}, and we only need to check that the head of $\widetilde\fromthe \widetilde P(w.\widehat\mu)$ is in grade $0$. But this is the case because
\begin{align*}
\hom_{\mathcal C_{\widehat{\lambda}}}(\widetilde\fromthe\widetilde P(w.\widehat\mu),\widetilde L(z.\widehat\lambda)\langle i\rangle)
\cong\hom_{\mathcal C_{\widehat\mu}}(\widetilde P(w.\widehat\mu),\widetilde L(z.\widehat\mu)\langle i\rangle)
\end{align*} is zero unless $i=0$.
\end{rem}


Fix $y,w\in W^J$ where $J$ is associated to $\mu\in\overline{C^-}$. Recall the filtration $U_i$ from \S \ref{sstr}. We still denote by $U_i$ the $\widetilde{\fg}$-module $\mathscr F_l^{-1}U_i$ embedded in $\mathbb O$. Our new definition of the quantum translation gives $U_0=\fromthe\Delta(y.\widehat{\lambda})$. Lemma \ref{max}, \ref{facetcrossinglem} and Corollary \ref{maxcor} are still valid. Using the graded translations, we construct a graded lift of $U_i$ starting from $\widetilde U_0=\widetilde\fromthe\widetilde\tothe\widetilde\Delta(y.\widehat\lambda)$. We have $$0\to\widetilde U_{i+1}\to \widetilde U_{i}\to \bigoplus_{x\in W_J,\  \ell(x)=i}\widetilde\Delta(yx.\widehat\lambda)\langle n_x \rangle \to 0,$$ for some $n_x\in \Z$ depending on $x$. In fact, we know what the shifts $n_x$ are:

\begin{prop}\label{gradedses}
The filtration $\{\widetilde U_i\}$ of  $\widetilde\fromthe\widetilde\tothe\widetilde\Delta(y.\widehat\lambda)$ satisfies the short exact sequences $$0\to\widetilde U_{i+1}\to \widetilde U_{i}\to \bigoplus_{x\in W_J,\  \ell(x)=i}\widetilde\Delta(yx.\widehat\lambda)\langle i\rangle \to 0$$ for all $i$.
\end{prop}
\begin{proof}
Since $\widetilde U_0$ has an irreducible head, its radical filtration agrees with its grading filtration by Koszulity. So this follows from Proposition \ref{radicalses}.
\end{proof}

\begin{cor}\label{filtrlinear}
For all $i$, $\widetilde U_i\langle -i\rangle\in\w{\mathcal{C_{\widehat\lambda}}}$ is linear.
\end{cor}
\begin{proof}
It follows by induction on $i$. The base case is proven in Corollary \ref{crossinglinear}, and Propositions \ref{gradedses}, \ref{linlem} do the induction step. 
\end{proof}

We need the following in order to apply Proposition \ref{linearpairty}.
\begin{lem}\label{lem}
For $\mu\in \overline{C^-_\Z}$ and $y,z\in W^+(\mu)$ with $\ell(y)\equiv \ell(z)$ mod $2$, we have 
$$\Ext_{\mathcal C_{\widehat\mu}}^1(L(y.\widehat\mu),L(z.\widehat\mu))=0.$$
\end{lem}

\begin{proof}
First note that the statement is true for a regular weight $\lambda$. (For example, it follows from \eqref{reghLCF}, its dual, and \cite[Corollary (3.6)]{cps1}.) 
Also, Koszulity implies that the radical filtration and the grade filtration of a standard module $\Delta(y.\widehat\lambda)$ are the same. 
So the grade filtration of $\Delta(y.\widehat\lambda)$ has alternating parity. Since $\widetilde\tothe$ is exact and preserves the parity of irreducibles, the module $\widetilde\tothe\widetilde\Delta(y.\widehat\lambda)=\widetilde\Delta(y.\widehat\mu)$ also has a grade filtration with alternating parity. Hence $\Delta(y.\widehat\mu)$ has a radical filtration with alternating parity. 
Now suppose 
$$0\to L(z.\widehat\mu)\to M\to   L(y.\widehat\mu)\to 0$$
represents a non-trivial element in $\Ext_{\mathcal C_{\widehat\mu}}^1(L(y.\widehat\lambda),L(z.\widehat\lambda))$. 
The linkage principle rules out any possibilities other than the cases $z>y$ or $y>z$. We may assume $y>z$ by duality. Then there is a surjective map from $\Delta(y.\widehat\lambda)$ to $M$. This contradicts the assumption that $z$ and $y$ are of the same parity and that $\Delta(y.\widehat\mu)$ has a radical filtration with alternating parity.
\end{proof}

We now obtain a key property of the modules $U_i\in \mathcal C_{\widehat{\lambda}}$. Recall that (for a general highest weight category $\mathcal C$) an object $M\in \mathcal C$ is said to have $N$-parity if $\Ext^{2n+1}_{\mathcal C}(M,N)=0$ for all $n\in \Z$ and is said to have parity if it has $L$-parity for all irreducible $L\in\mathcal C$.

\begin{cor}\label{keycor}
For each $i$, $U_i$ has parity. 
\end{cor}
\begin{proof}
This is an immediate corollary of Corollary \ref{filtrlinear}, Lemma \ref{lem}, and Proposition \ref{linearpairty}.
\end{proof}

\begin{ex}\label{cofiltlinpar}
Consider the quotient $\widetilde U_i':=\widetilde U_0/\widetilde U_i$ of $\widetilde U_0$. We have $$\widetilde\fromthe\widetilde\tothe\Delta(y.\widehat\lambda)=\w U'_N\surj \w U'_{N-1}\surj \cdots\surj\w U'_1 \twoheadrightarrow\w U'_0=0,$$ where $N=\ell(w_J)$. By Corollary \ref{nonlin} $\widetilde U_i'$ is not linear, even up to shift, for $1<i<N$, while $U'_i=F(\w U'_i)$ has parity if $i$ is odd. (If $i$ is odd, then $U_i$ has $L$-parity opposite of $U_0$ with respect to any irreducible $L$. Lemma \ref{2/3} shows that $U'_i$ has $L$-parity.)
\end{ex}

\subsection{Cohomology in singular blocks}\label{main}
We are ready to prove our main theorem using that $U_i$ has parity. Note that the statement of Corollary \ref{keycor} does not involve any grading. We now forget the grading and prove our main theorem. 
Recall the definition $$\bar{P}^{J}_{y,w}=\sum_{x\in W_J}(-1)^{\ell(x)}\bar P_{yx.w}.$$

\begin{thm}\label{thm}\cite[Conjecture III]{PS14}
Suppose $l$ is KL-good for the root system $R$. Let $\mu\in \overline{C^-_\Z}$ and $J=\{s\in S_l\ |\ s.\mu=\mu\}$. We have $$\sum_{n=0}^\infty \dim \Ext_{U_\zeta}^n(\Delta(y.\mu),L( w.\mu))t^n=t^{\ell( w)-\ell( y)}\bar{P}^{J}_{y,w}$$ for $y,w\in W^J$. 
\end{thm}

\begin{proof}
As we discussed in the beginning of \S \ref{koszul}, this follows if we show $$\sum_{n=0}^\infty \dim \Ext_{\widehat{\fg}}^n(\Delta(y.\widehat\mu),L( w.\widehat\mu))t^n=t^{\ell( w)-\ell( y)}\bar{P}^{J}_{y,w}$$ for $y,w\in W^J$. 

We first reduce the statement to the case where the assumptions in \S\ref{sssvv} are satisfied. 
If we pick a large integer $l'\geq h$ that is divisible by $2D$, there is a regular weight $\widehat\lambda$ and a weight $\widehat\nu$ of level $k'$ with $k'=-l'/2D-g\in \Z$ such that the integral Weyl group of $\widehat\lambda,\widehat\nu$ are both isomorphic to $W_{l'}$ and $\stab_{W_{l'}}(\widehat\nu)$ is isomorphic to $\stab_{W_l}(\widehat \mu)$ under the Coxeter group isomorphism $(W_l,S_l)\xrightarrow{\sim} (W_{l'},S_{l'})$. By Fiebig's combinatorial description \cite[Theorem 11]{Fiecomb}, it is enough to prove the theorem for $\widehat \nu$ instead of $\widehat \mu$. The problem of the full category $\mathbb O$ in \cite{Fiecomb} and the categories of \cite{SVV} being different is treated in \cite{PS15}.\footnote{In \cite{PS15}, it is similarly shown that $U_\zeta$-mod is Koszul. Using that we could have worked entirely in the quantum case to prove the theorem. But then, if $l<h$, there is no regular weight we can translate from, and we will anyway have to use the affine category $\mathcal O$ to obtain our result for small (KL-good) $l$.} So we may assume that we are in the situation in \S\ref{sssvv}.

Let $\widehat\lambda$ be a regular weight. We translate from $\widehat\lambda$ to $\widehat\mu$ as in \S\ref{sssvv}. Corollary \ref{filtrlinear} and Proposition \ref{linearpairty} show that each $U_i$ has parity. In particular it has $L=L( w.\widehat\lambda)$-parity, that is, $\Ext_{\widehat{\fg}}^n(U_i,L)$ is zero in every other degree. To be more precise, $U_i$ is $L$-even (resp., odd), if and only if $\bigoplus_{\ell(x)=i, x\in W_J}\Delta (yx.\widehat\lambda)$ is $L$-even (resp., odd), if and only if $U_{i+1}$ is $L$-odd (resp., even).
Therefore, half the terms in the long exact sequence induced by applying $\Hom_{\widehat{\fg}}(-,L)$ to each short exact sequence 
$$0\to U_{i+1}\to U_i\to \bigoplus_{\ell(x)=i, x\in W_J}\Delta (yx.\widehat\lambda)\to 0$$
vanish, and the sequence splits into the short exact sequences 
$$0\to\Ext_{\widehat{\fg}}^{n-1}(U_{i+1},L)\to\Ext_{\widehat{\fg}}^{n}(\bigoplus_{\ell(x)=i}\Delta (yx.\widehat\lambda),L)\to\Ext_{\widehat{\fg}}^{n}(U_{i},L)\to 0.$$
They give
\begin{align*}
\sum_{n=0}^\infty \dim \Ext_{\widehat{\fg}}^n(U_{i},L)t^n
&=\sum_{n=0}^\infty \dim \Ext_{\widehat{\fg}}^n(\oplus_{\ell(x)=i}\Delta (yx.\widehat\lambda),L)t^n\\
&\ \ \ \ -t\sum_{n=0}^\infty \dim \Ext_{\widehat{\fg}}^n(U_{i+1},L)t^n
\end{align*} for all $n$.

Putting them together, we get

\begin{align*}
\sum_{n=0}^\infty \dim \Ext_{\widehat{\fg}}^n(\Delta(y.\widehat\mu),L(w.\widehat\mu))t^n
&=\sum_{n=0}^\infty \dim \Ext_{\widehat{\fg}}^n(\fromthewall\tothewall\Delta(y.\widehat\lambda),L(w.\widehat\lambda))t^n\\
&=\sum_i (-t)^i\sum_{n=0}^\infty \dim \Ext_{\widehat{\fg}}^n(\oplus_{\ell(x)=i}\Delta (yx.\widehat\lambda),L(w.\widehat\lambda))t^n\\
&=\sum_i (-t)^i\sum_{\ell(x)=i}\sum_{n=0}^\infty \dim \Ext_{\widehat{\fg}}^n(\Delta (yx.\widehat\lambda),L(w.\widehat{}\lambda))t^n\\
&=\sum_{x\in W_J}(-t)^{\ell(x)}\sum_{n=0}^\infty \dim \Ext_{\widehat{\fg}}^n(\Delta (yx.\widehat\lambda),L(w.\widehat\lambda))t^n \\
&=\sum_{x\in W_J}(-t)^{\ell(x)}t^{\ell(w)-\ell(yx)}\bar{P}_{yx,w}\\
&=t^{\ell(w)-\ell(y)}\sum_{x\in W_J}(-1)^{\ell(x)}\bar{P}_{yx,w}\\
&=t^{\ell(w)-\ell(y)}\bar P^{J}_{y,
w},
\end{align*}
and we are done.
  \end{proof}

For the next corollary, we make statements in $U_\zeta$-mod rather than in $\mathbb{O}$ or in $\mathcal O$ in order to simplify the notation. In particular, $U_i$ is in $U_\zeta$-mod again. In Theorem \ref{thm}, we computed the dimensions of $\Ext^n_{U_\zeta}(U_i,L(w.\lambda))$, for $w\in W^+(\mu)$. But we don't need $w$ to be in $W^+(\mu)$:
\begin{cor}
Fix an integer $i$. We have for $y\in W^+(\mu)$ and $w\in W^+$, $$\sum_{n=0}^\infty \dim \Ext_{U_\zeta}^n(U_i,L(w.\lambda))t^n=t^{\ell(w)-i}\sum_{x\in W_J,\ \ell(x)\geq i}(-1)^{\ell(x)-i}{P}_{yx,w}.$$ In particular, this polynomial has non-negative coefficients.
\end{cor}
\begin{proof}
Since all $U_j$ have $L(w.\lambda)$-parity, we obtain the formula as in the proof of Theorem \ref{thm}.
\end{proof}

If $w\not\in W^J$, then $\fromthe L(w.\mu)$ is $0$ and $\Ext_{U_\zeta}^n(U_0,L(w.\lambda))$ is $0$. This shows an identity in Kazhdan-Lusztig polynomials (which might have been known for any $y\in W_l$ and $w\not\in W^J$).
\begin{cor}
If $w\in W^+\setminus W^+(\mu)$ and $y\in W^+(\mu)$, then $$\sum_{x\in W_J}(-1)^{\ell(x)}{P}_{yx,w}=0.$$
\end{cor}

\subsection{Graded enriched Grothendieck groups} \label{groth}
We present another proof of Theorem \ref{thm}. We are still in the setting of \S \ref{sssvv}. In particular, $w \in W^J$. Our plan is to apply the translation functor $\widetilde\tothe:\wC_{\widehat{\lambda}}\to\wC_{\widehat{\mu}}$ to a sequence of distinguished triangles that realizes $\widetilde L(w.\widehat\lambda)$ in $\wE^R(\wC_{\widehat{\lambda}})$. Recall the construction in the proof of Proposition \ref{recog}. Replacing $\nabla(\widehat\lambda_i)[n_i]$ by $\widetilde\nabla(\lambda_i)\{n_i\}$, we obtain the graded complexes $\widetilde L(w.\widehat\lambda)=Y_0,\cdots,Y_N=0$ in $\wE^R$. Writing $\lambda_i=w_i.\widehat\lambda$, there is a distinguished triangle $$Y_{i+1}\to Y_i\to \widetilde\nabla(w_i.\widehat\lambda)\{n_i\}\to$$ for each $0\leq i\leq N$. We know by Lemma \ref{lem} and Proposition \ref{linearpairty} that $n_i\equiv \ell(w)-\ell(w_i)$ mod 2. 
Since the translation functors are exact, applying $\w\tothe$ to the sequence produces the sequence $\widetilde L(w.\widehat{\mu})=\w\tothe Y_0,\cdots,\w\tothe Y_N=0$ of objects in $\mathcal D^b(\wC_{\widehat\mu})$ and distinguished triangles $$\w\tothe Y_{i+1}\to \w\tothe Y_i\to \w\tothe\widetilde\nabla(w_i.\widehat\lambda)\{n_i\}\to$$ in $\mathcal D^b(\wC_{\widehat\mu})$.

\begin{prop}\label{grtothe}
 We have $$\w\tothe\w\nabla(yx.\widehat\lambda)\cong\w\nabla(y.\widehat\mu)\langle -\ell(x)\rangle,\ \  \w\tothe\w\Delta(yx.\widehat\lambda)\cong\w\Delta(y.\widehat\mu)\langle \ell(x)\rangle$$ for $y\in W^J, x\in W_J$.
\end{prop}
\begin{proof}
We show only the assertion for $\w\Delta(yx.\widehat\lambda)$. Let $\ell(x)=i$. Recall that $$\w\tothe\w L(yx.\widehat\lambda)\cong \delta_{yx,y}\w L(y.\widehat\lambda).$$ Since $\tothe\Delta(yx.\widehat\lambda)\cong\Delta(y.\widehat\mu)$ and $\Delta(y.\widehat\mu)$ has only one composition factor isomorphic to $L(y.\widehat\mu)$, it is enough to show that $\w\Delta(yx.\widehat\lambda)$ has $\w L(y.\widehat\lambda)\langle i\rangle$ as its composition factor. By the Humphreys-Verma reciprocity, 
this is equivalent to $\w\Delta(yx.\widehat\lambda)\langle i\rangle$ appearing in a $\w\Delta$-filtration of $\w P(y.\widehat\lambda)$. 
But we saw in Proposition \ref{radicalses} that this is true for $\w U_0$ instead of $\w P(y.\widehat\lambda)$, because Koszulity implies that the radical filtration of $\w U_0$ agrees with the grading filtration. 
Since $\w P(y.\widehat\lambda)\surj \w U_0$, and since the kernel of this map has a $\Delta$-filtration, this is enough.
\end{proof}

Writing $w_i=y_ix_i$ with $y_i\in W^J,x_i\in W_J$ uniquely, Proposition \ref{grtothe} tells us that the distinguished triangles are $$\w\tothe Y_{i+1}\to \w\tothe Y_i\to \widetilde\nabla(y_i.\widehat\mu)\{n_i\}\langle-\ell(x_i)\rangle=\widetilde\nabla(y_i.\widehat\mu)[\ell(x_i)]\{ n_i-\ell(x_i)\}\to.$$ These are not distinguished triangles in $\wE^R$. But we know by Theorem \ref{svv} that there exists a sequence $\widetilde L(w.\widehat\mu)=X_0,\cdots,X_{N'}=0$ in $\wE^R$ with distinguished triangles $$X_{j+1}\to X_j\to \widetilde\nabla(z_j.\widehat\mu)\{m_j\}\to.$$ Let us compare these two sequences to determine the (unordered) multiset $\{(z_j,m_j)\}$.

 Consider the enriched Grothendieck group $K^R=K^R_0(\mathcal C_{\widehat\mu})$ and the graded enriched Grothendieck group $\w K^R=K^R_0({\wC_{\widehat\mu}})$ defined in \cite{cps1}.  The two sequences provide two expressions of $[\w L(w.\widehat\mu)]\in\w K^R$ with respect to the $\Z[t,t^{-1}]$-basis $\{[\w\nabla(y.\widehat\mu)]\}_{y\in W^J}$.
The sequence $\w\tothe Y_i$ provides 
\begin{equation}\label{yexpression}
\sum_{0\leq i\leq N}(-1)^{\ell(x_i)}t^{n_i-\ell(x_i)}[\w\nabla(y_i.\widehat\mu)],
\end{equation} and the sequence $X_j$ provides $$\sum_{0\leq j\leq N'}t^{m_j}[\w\nabla(z_j.\widehat\mu)].$$ Let $c_{y,n}$ be the $\Z$-coefficient of $t^{-n}[\w\nabla(y.\widehat\mu)]$ in the expression, thus $$c_{y,n}=|\{j\in[0,N']\ |\ z_j=y,\ -m_j=n\}|.$$ (Recall that $m_j$ are negative integers.) This is the dimension of $\ext^n_{\mathcal C_{\mu}}(\w\Delta(y.\widehat\mu)\langle -n\rangle,\w L(w.\widehat\mu))$ which is the same as $\Ext^n_{\mathcal C_{\mu}}(\Delta(y.\widehat\mu),L(w.\widehat\mu))$ by standard Koszulity. 

The expression \eqref{yexpression} determines $c_{y,n}$. It remains to write down the relation explicitly. We have \begin{align*}
c_{y,n}
=|\{i\in [0,N] &|\ y_i=y,\ n_i-\ell(x_i)=-n ,\ \ell(x_i) \textrm{ even} \} |\\&-|\{i\in [0,N] |\ y_i=y,\ n_i-\ell(x_i)=-n ,\ \ell(x_i) \textrm{ odd} \} |.
\end{align*} 
Letting $c^x_{y,n}:=|\{i\in [0,N]\ | y_i=y,\ x_i=x,\ -n_i=n \}|$, we can write
$$c_{y,n}=\sum _{x\in W_J}(-1)^{\ell(x)}c^x_{y,n-\ell(x)}.$$ Note also that $$c^x_{y,n}=|\{i\in [0,N]\ |w_i=yx,\ -n_i=n \}|.$$ Since we started from the realization $Y_i$ of $\w L(w.\widehat\lambda)$, the number $c^x_{y,n}$ is the dimension of $$\ext^n_{\mathcal C_{\widehat\lambda}}(\w\Delta(yx.\widehat\lambda)\langle -n\rangle,\w L(w.\widehat\lambda)) \cong\Ext^n_{\mathcal C_{\widehat\lambda}}(\Delta(yx.\widehat\lambda),L(w.\widehat\lambda)).$$

Combining all this, we obtain the identity
$$\dim\Ext^n_{\mathcal C_{\widehat\lambda}}(\Delta(y.\widehat\mu),L(w.\widehat\mu))=\sum_{x\in W_J}\dim\Ext^{n-\ell(x)}_{\mathcal C_{\widehat\lambda}}(\Delta(yx.\widehat\lambda),L(w.\widehat\lambda)).$$ This is equivalent to the formula \eqref{hLCF} by the formula \eqref{reghLCF}. Finally, we transfer this to the quantum case as in the first proof.  

\subsection{Ext-groups between irreducibles}
 Dualizing Theorem \ref{thm}, we obtain $$\sum_{n=0}^\infty \dim \Ext_{U_\zeta}^n(L( w.\mu),\nabla(y.\mu))t^n=t^{\ell( w)-\ell( y)}\bar{P}^{J}_{y,w}$$ for $y,w\in W^J$. Then \cite[Corollary (3.6)]{cps1} combined with the fact that $P^{J}_{y,w}$ is a polynomial on $t^2$ shows that the dimension for $\Ext_{U_\zeta}^\bullet(L(w.\mu),L(z.\mu))$ is given as $$\dim\Ext_{U_\zeta}^n(L(w.\mu),L(z.\mu))=\sum_{i+j=n, y\in W^+(\mu)}\dim\Ext^i_{U_\zeta}(L(w.\mu),\nabla(y.\mu))\dim\Ext_{U_\zeta}^j(\Delta(y.\mu),L(z.\mu)).$$ This is a finite sum as the right hand side is $0$ unless $y\leq w,z$. We have proved the following.
\begin{thm}\label{irred}
Suppose $l$ is KL-good. Let $\mu\in\overline {C^-_\Z}$, $J=\{s\in S_l\ |\ s.\mu=\mu\}$, and $w,z\in W^+(\mu)$. Then we have
$$\sum_{n=0}^\infty \dim \Ext_{U_\zeta}^n(L( w.\mu),L( z.\mu))t^n=
\sum_{ y\in W^+(\mu)}t^{\ell( w)+\ell( z)-2\ell( y)}\bar{P}^{J}_{ y, w}\bar{P}^{J}_{ y, z}.$$
\end{thm}

\subsection{Cohomology for $q$-Schur algebras}
The above results provide calculations of $\Ext$-groups
between irreducible modules for important families of finite dimensional algebras associated to quantum enveloping
algebras.

Consider first the type A quantum groups $U_\zeta(\mathfrak{sl_n})$. Any positive integer $l$ is KL-good in this case. As explained in \cite[\S 9]{pskoszulproperty}, a classical $q$-Schur algebra over $\C$ with $q=\zeta^2$ arises as a truncation of $U_\zeta(\mathfrak{sl_n})$-mod by a certain ideal $\Gamma$ of dominant weights. Thus, Theorem \ref{thm} and Theorem \ref{irred} compute the corresponding cohomology for $q$-Schur algebras.

A generalized $q$-Schur algebra arises in a similar way.
Let $X’$ be the union of a
(finite) collection of $W_l$ linkage classes in $X^+$, regarded as a poset using the dominance order $\uparrow$ on
$X^+$. Let $\Gamma$ be a finite ideal in $X’$, and let $\mathcal C^\zeta[\Gamma]$ be the Serre subcategory of
$C_\zeta=U_\zeta$-mod generated by the irreducible modules $L(\gamma)$, $\gamma\in\Gamma$. Then
$\mathcal C^\zeta[\Gamma]$ is equivalent to the category $A_\Gamma$-mod of finite dimensional modules for some finite dimensional
algebra $A_\Gamma$. The algebra $A_\Gamma$ is only determined up to Morita equivalence. But, by abuse of language, it is
often called ``the generalized $q$-Schur algebra" associated to $\Gamma$. This defines the generalized $q$-Schur algebras for all other types as well. Now, in any type (assuming $l$ KL-good), Theorem \ref{thm} and Theorem \ref{irred} provide the corresponding cohomology dimension for the generalized $q$-Schur algebras.  
\bibliographystyle{abbrv}
\def\cprime{$'$} \def\cprime{$'$} \def\cprime{$'$}

\end{document}